\documentclass[preprint, 12pt]{elsarticle}
\makeatletter
\def\ps@pprintTitle{%
 \let\@oddhead\@empty
 \let\@evenhead\@empty
 \def\@oddfoot{}%
 \let\@evenfoot\@oddfoot}
\makeatother
\usepackage{hyperref}
\usepackage{appendix}
\usepackage[utf8]{inputenc}
\usepackage{amsfonts}
\usepackage{amsmath}
\usepackage{bbm}
\usepackage[thmmarks,  thref, amsmath]{ntheorem}
\usepackage{amssymb}
\usepackage{graphicx}
\usepackage{enumerate}
\usepackage{mathtools}
\newtheorem{theorem}{Theorem}
\newtheorem{lemma}{Lemma}
\newtheorem{corollary}{Corollary}
\newtheorem{remark}{Remark}

\theoremstyle{nonumberplain}
\theoremheaderfont{\scshape}
\theorembodyfont{\upshape}
\theoremheaderfont{\itshape}
\theorembodyfont{\upshape}

\theoremstyle{nonumberplain}
\theoremheaderfont{\scshape}
\theorembodyfont{\upshape}
\theoremsymbol{$\square$}
\theorempostwork{\setcounter{case}{0}}

\theoremheaderfont{\itshape}
\theorembodyfont{\upshape}
\newtheorem{proof}{Proof:}

\newcommand{\R}{\mathbb{R}}

\newcommand{\mC}{\mathcal{C}}

\DeclareMathOperator\supp{supp}

\usepackage[nocompress,sort]{cite}
\usepackage{lineno}

\begin{document}

\begin{frontmatter}

\title{The Range Description of a Conical Radon Transform}

\author[main_address]{Weston Baines\corref{corresponding_author}}
\cortext[corresponding_author]{Corresponding author}\ead{bainesw1@tamu.edu}
\address[main_address]{Department of Mathematics
Mailstop 3368
Texas A\&M University
College Station, TX 77843-3368}

\begin{abstract}
In this work we consider the Conical Radon Transform, which integrates a function on $\R^n$ over families of circular cones. Transforms of this type are known to arise naturally as models of Compton camera imaging and single-scattering optical tomography (in the latter case, when $n=2$). The main results (which depend on the parity of $n$) provide a description of the range of the transform on the space $C_0^\infty(\R^n)$
\end{abstract}


\end{frontmatter}

\nocite{*}

\section{Introduction}\label{S:intro}
In his seminal 1923 paper, Arthur Compton derived a physical model for the phenomenon by which X-rays scatter via interaction with charged particles \citep{Compton_1923}. This phenomenon has come to be known as \emph{Compton scattering}. The scattering interaction between a high-energy photon and an electron is modeled by the equation
\begin{equation} \label{eq: Compton_Scattering}
    E_{\gamma'} = \frac{E_\gamma}{1+(E_\gamma/(m_e c^2))(1-\cos \theta)},
\end{equation}
(see Fig. \ref{fig: Compton_Scattering}) where $E_\gamma$ is the initial energy of the photon, $E_{\gamma'}$ is the energy of the photon after scattering, $m_e$ is the rest mass of an electron, $c$ is the speed of light, and $\theta$ is the scattering angle.
\begin{figure}[ht]
    \centering
    \includegraphics[width=0.7\textwidth]{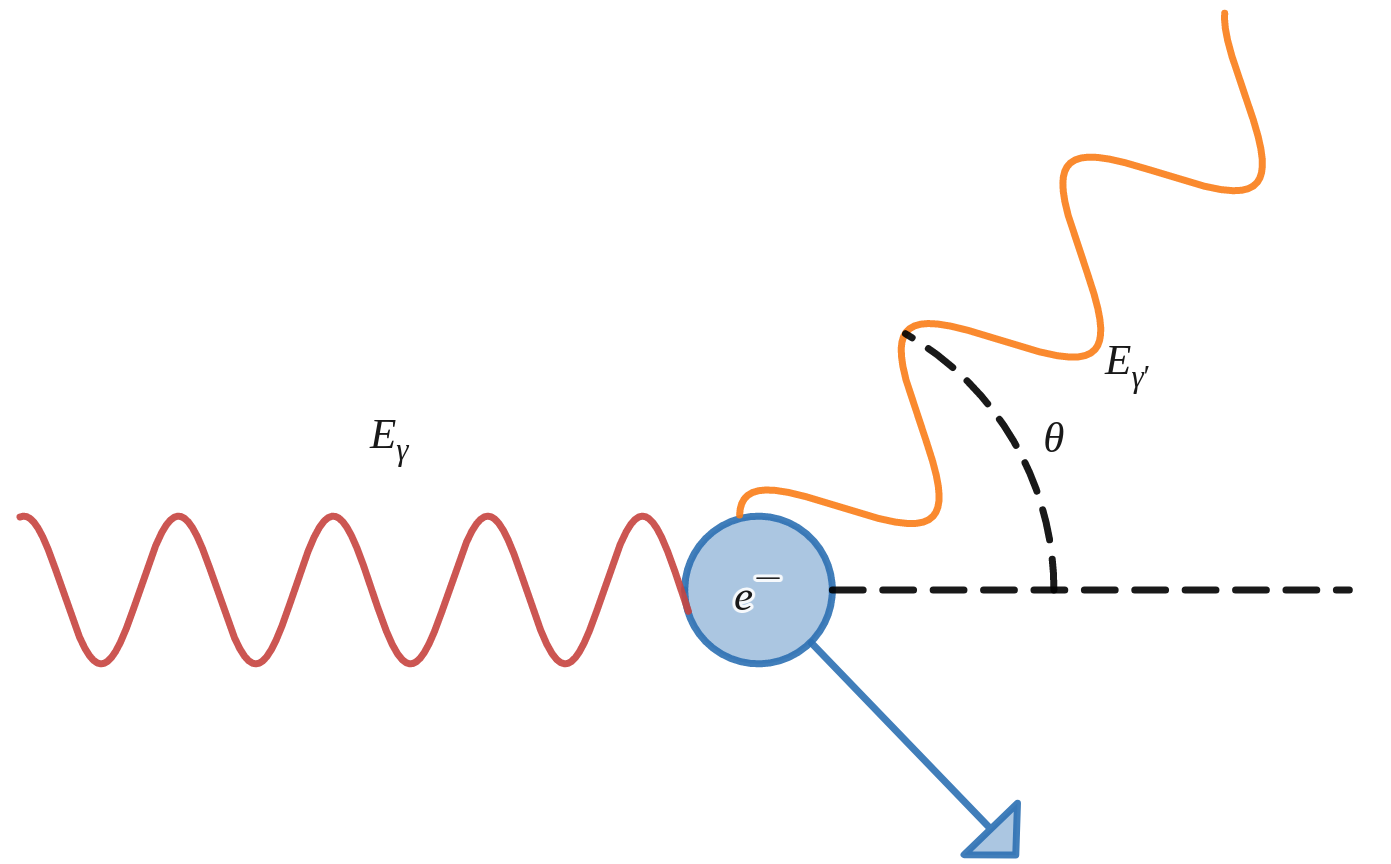}
    \caption{Incident X-ray scatters off electron at angle $\theta$ via Compton Scattering, transferring $E_{\gamma}-E_{\gamma'}$ energy to the electron.}
    \label{fig: Compton_Scattering}
\end{figure}

\begin{figure}[ht!]
    \centering
    \includegraphics[width=0.7\textwidth]{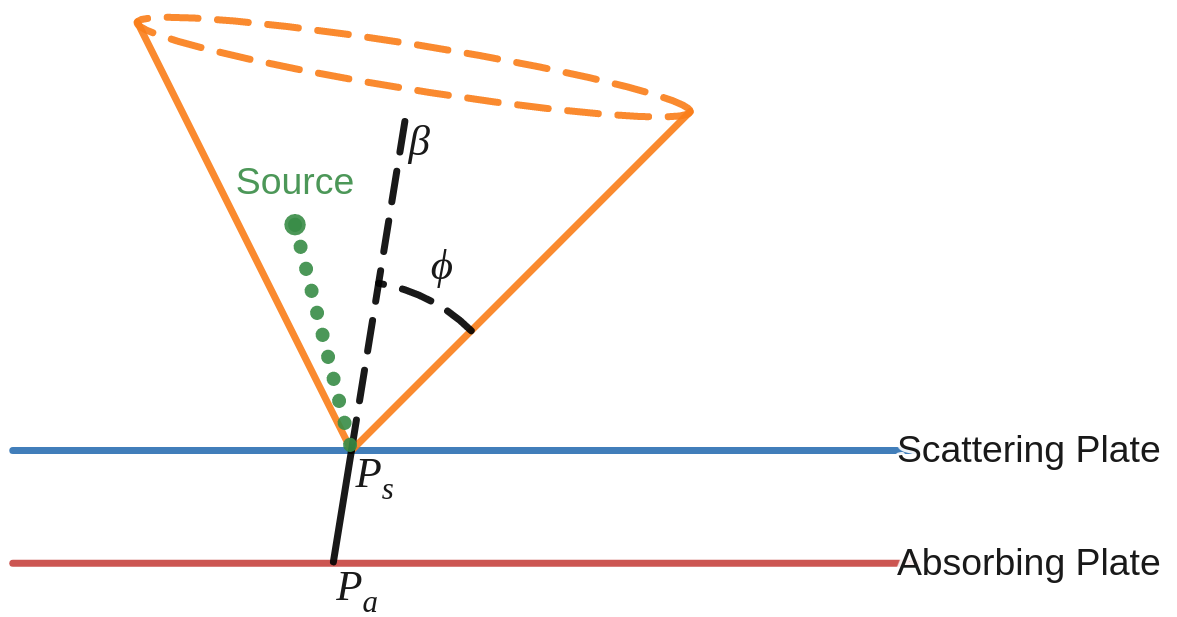}
    \caption{Schematic representation of a Compton camera.}
    \label{fig: Compton_Camera}
\end{figure}

The recently popularized Compton camera utilize Compton scattering to detect high energy photons. A Compton cameras consists of a pair of parallel detector plates (See Fig. \ref{fig: Compton_Camera}). When a high energy photon hits the first plate, it undergoes Compton scattering and the scattered photon is absorbed in the second plate. The positions $\boldsymbol P_s$ and $\boldsymbol P_a$ of the scattering and absorption events, as well as the deposited energies $E_s$ and $E_a$ are recorded. This data allows one to determine a surface cone
\begin{equation}
    \mathfrak{C}(\boldsymbol P_s, \phi, \boldsymbol \beta)=\{ \boldsymbol y \in \mathbb{R}^n : (\boldsymbol y - \boldsymbol P_s) \cdot \boldsymbol \beta = |\boldsymbol y -\boldsymbol P_s| \cos \phi \}.
\end{equation}
containing the incident photon trajectory (see Fig. \ref{fig: Compton_Camera}). This cone has vertex $\boldsymbol P_s$, opening (scattering) angle $\phi$ and central axis direction $\boldsymbol \beta$ determined by
\begin{equation}
    \cos \phi = 1- \frac{m_e c^2 E_s}{(E_s+E_a)E_a} \qquad \boldsymbol \beta = \frac{\boldsymbol P_s-\boldsymbol P_a}{|\boldsymbol P_s-\boldsymbol P_a|}.
\end{equation}

In this manuscript we will assume that high energy $\gamma$-photons are emitted from a radiating source with intensity source distribution $f$ that is smooth and compactly supported ($f\in C_{0}^{\infty}(\mathbb{R}^n)$). By exposing the Compton camera to $\gamma$-photons for a sufficiently long duration of time and counting the number of photons measured at each position $\boldsymbol x$ through a scattering angle $\phi$ and with central cone axis $\boldsymbol \beta$ we can approximate the integral of $f$ over surface cones:
\begin{equation} \label{eq: CRT}
    \mathcal{C}[f](\boldsymbol x, \phi, \boldsymbol \beta) = \int_{\mathfrak{C}(\boldsymbol x, \phi, \boldsymbol \beta)} f(\boldsymbol y)  dS(\boldsymbol y),
\end{equation}
where $dS$ is the surface measure of the cone.

This is a Radon type transform, as it projects functions onto a parametric family of hypersurfaces. For this reason many works refer to this as the \emph{Conical Radon Transform} (CRT) \citep{GZ,Ambartsoumian2013,Palamodov2017}, as will we in this manuscript. Depending on the specific engineering of the Compton camera, the measure $dS$ may include additional weight terms, typically a power weight $|x-y|^{-k}$ \citep{Basko,Maxim,Smith2005}. We denote this \emph{Weighted Conical Radon Transform} as $\mathcal{C}^{k}$. When $k<n-1$ the transform $\mathcal{C}^{k}$ is called \emph{regular} and when $k=n-1$, it is called \emph{singular} \citep{Palamodov2017}. 

The main advantage of Compton cameras over other traditional $\gamma$-photon detectors, such as Anger cameras, is that they give directional information without any need for collimation. In order for an Anger camera to detect trajectories of $\gamma$-photons, a mechanical filter must be placed in front of the detector to block all photons except those traveling along a particular path \citep{Peterson} (see also Fig. \ref{fig: Collimation}). Although this gives very precise directional information, the signal is significantly attenuated, making it unsuitable in cases of weak signals and strong background noise.

\begin{figure}[ht!]
    \centering
    \includegraphics[width=0.7\textwidth]{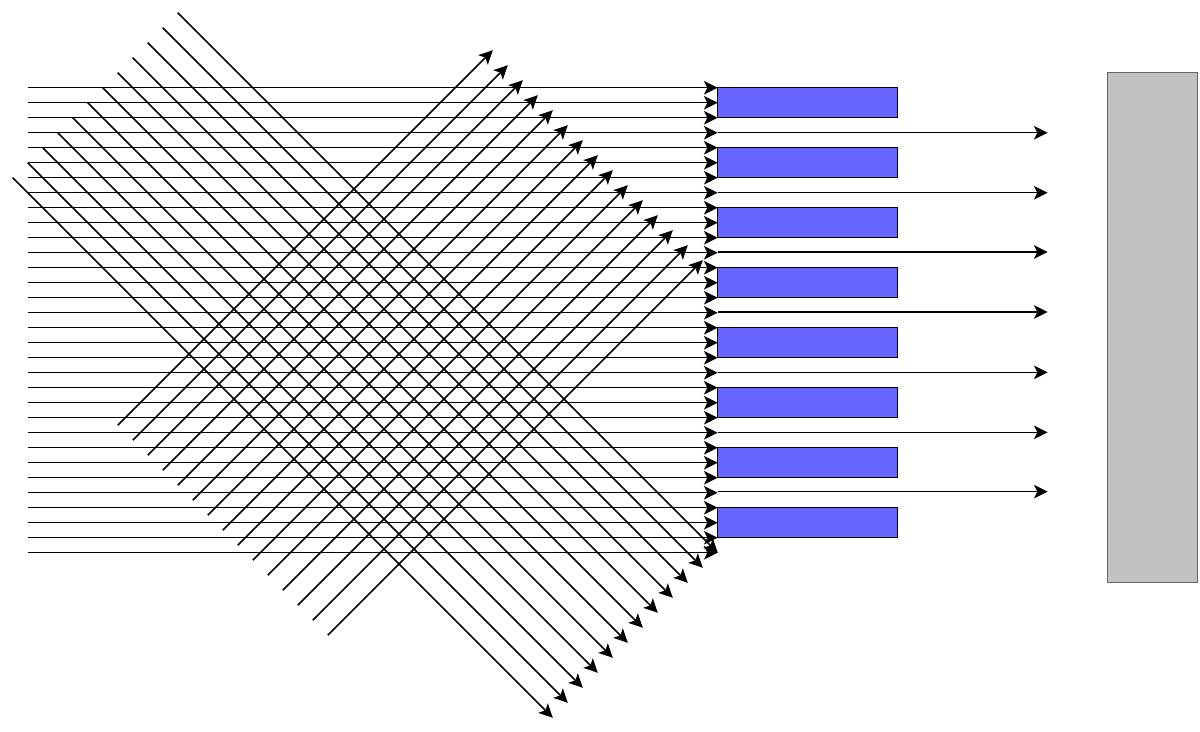}
    \caption{Schematic representation of mechanical collimation in an Anger camera.}
    \label{fig: Collimation}
\end{figure}

An important application where traditional Anger cameras are not feasible is detecting the presence of illicit nuclear material \citep{ADHKK, cargo}. In such settings the nuclear source is shielded and very weak, and the background noise is very strong, with a signal-to-noise (SNR) ratio significantly less than 1\%. The difficulty is further compounded by the presence of complex configurations of scattering and absorbing materials.

The primary goal of Compton imaging is to recover $f$ from $\mathcal{C}[f]$, that is, inversion of the CRT. Inversion of Radon-type transforms is a rich area of study in inverse problems, where one must address questions of existence, uniqueness and stability of solution \citep{Kuchment,Natt}. The CRT is particularly interesting since the hypersurfaces under consideration (surface cones) have a singularity, and the transform is overdetermined (the CRT maps a function of $n$ variables to a function of $2n$ variables). The CRT, being overdetermined, has ``small'' range - in the sense that it has infinite co-dimension - thus, there are infinitely many left inverses, and in fact a variety of inversion formulae have been studied (see \citep{TKK} and the references therein). 

An important topic in tomographic studies of Radon-type transforms is the description of their ranges \citep{Kuchment,Natt}, since they aid in improving inversion algorithms, completing incomplete data, correcting measurement errors, measuring sampling errors, etc. \citep{Natt,Fatma}. 



\begin{figure}[ht!]
    \centering
    \includegraphics[width=0.6\textwidth]{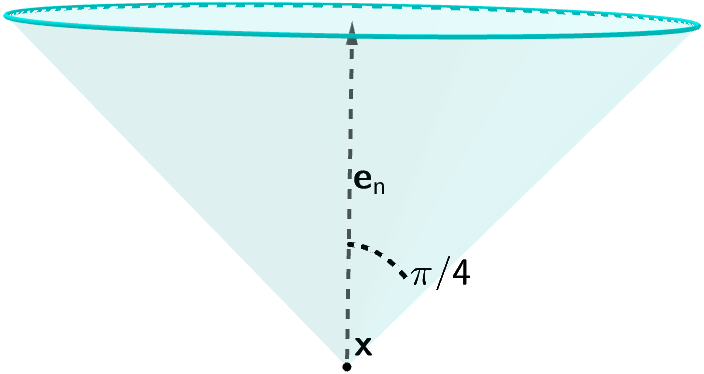}
    \caption{3D cone with central axis aligned along the ``vertical'' coordinate and $\pi/2$ radians opening.}
    \label{fig: Ex_Axis_Aligned_Cone}
\end{figure}



This manuscript is structured as follows. Section \ref{S:notation} contains the main notions and notations involved. In Section \ref{S:results} we describe the range of the CRT under the restriction of fixed cone axis on the space $C_0^\infty(\R^n)$ of smooth functions with compact support. We find that the description differs for even and odd dimensions, which one might expect, as the wave equation is relevant in the study of the CRT (see e.g. the discussion in Sec. 5 of \citep{Palamodov2017}). Proofs of these results are given in Section \ref{S:proofs}. These results are formulated and stated, to avoid complicating notations, for cones with the opening angle $\pi/2$ radians (half-opening angle $\pi/4$ radians). They, however, can be readily formulated for any opening angle, which is done in Section \ref{S:angle}. Since even restricted CRT data is sufficient for inversion, one might expect that studying symmetries of the CRT will reveal the full range description. Indeed, we show that this is possible in Section \ref{S: Full_Range}.

\section{Some Objects and Notations}\label{S:notation}
Throughout the following sections, we will denote vectors with a bold font, e.g. $\boldsymbol x \in \mathbb{R}^{n-1}$ and $\boldsymbol \omega \in \mathbb{C}^{n-1}$. For a complex number $z$ we denote its real and imaginary parts with $\Re{(z)}$ and $\Im{(z)}$ respectively.


When we discuss the restricted CRT, we will write $\R^n=\R^{n-1}_{\boldsymbol x}\times\R_t$, and thus vectors are represented as $(\boldsymbol x,t)$. We will assume that the axes of all cones are aligned with the $t$-direction, i.e. 
$$\boldsymbol \beta=\boldsymbol e_n=(\textbf{0},1) \in \R^{n-1}_{\boldsymbol x} \times \R_t.$$
For a fixed opening angle $\phi$ will denote the restricted CRT with
\begin{equation}
    \mathcal{C}_{\phi}[f] (\boldsymbol x, t) = \mathcal{C}[f]((\boldsymbol x,t),\phi,\boldsymbol e_n).
\end{equation}
In this case we can identify the CRT of a function as its convolution with the distribution
\begin{equation}
    D_\phi(\boldsymbol x,t) = \delta(-t-|(\cot \phi) \boldsymbol x|) 
\end{equation}
where $\delta$ is the Dirac delta distribution.\footnote{
Appearance of such cones suggests a possible difference in formulas of the CRT depending on the parity of the dimension. This difference does materialize. E.g., for odd dimensions inversion require non-local transformation.}

The weighted CRT will be important for our study of the range of the CRT, and under the restrictions on opening angle and axis direction, the weighted CRT of a function can be written as its convolution with the distribution
\begin{equation}
    D_{w,\phi} (\boldsymbol x, t,\phi) = w(\boldsymbol x,t,\phi) \delta(-t-|(\cot \phi) \boldsymbol x|)
\end{equation}
where $w(\boldsymbol x,t,\phi)$ is a given weight. While a variety of weight functions can arise, in this work we will only need the power weight $w(\boldsymbol x,t,\phi)=|((\cot \phi) \boldsymbol x,t)|^{-1}$. We will denote the weighted CRT with power weight 
\begin{equation}
    D_{\phi}^{1} \coloneqq D_{|((\cot \phi) \boldsymbol x,t)|^{-1}} = |((\cot \phi) \boldsymbol x,t)|^{-1}\delta(-t-|(\cot \phi) \boldsymbol x|)
\end{equation}
and the corresponding weighted CRT as
\begin{equation}
    \mathcal{C}_{\phi}^{1}[f] = D_{\phi}^{1} \ast f.
\end{equation}
For the special case discussed in Sections \ref{S:results} and \ref{S:proofs} we will frequently omit the subscript $\phi$ in the notation as we fix $\phi=\pi/4$:
\begin{equation*}
    \mathcal{C}[f] (\boldsymbol x, t) = \mathcal{C}[f]((\boldsymbol x,t),\pi/4,\boldsymbol e_n), \qquad D(\boldsymbol x,t) = \delta(-t-|\boldsymbol x|).
\end{equation*}
\begin{equation*}
    \mathcal{C}^{1}[f] (\boldsymbol x, t) = |(\boldsymbol x,t)|^{-1}\delta(-t-|\boldsymbol x|) = D^{1} \ast f.
\end{equation*}
We also introduce the standard d'Alembertian operator
\begin{equation}
    \square  \coloneqq \frac{\partial^2}{\partial t^2} - \Delta_{\boldsymbol x},
\end{equation}
where $\Delta_{\boldsymbol x}$ is the Laplacian with respect to the spatial variable $\boldsymbol x$.

The Heaviside function will be denoted as 
$$h(t) \coloneqq \Bigg\{\begin{array}{cc}
           1, & t\geq 0 \\
          0 & t<0 
         \end{array}
$$
and
\begin{equation}
    \Theta(\boldsymbol x,t): = h(t) \cdot \delta (\boldsymbol x).
\end{equation}
The distribution $\Theta$ is supported on a ray, and if we convolve $\Theta$ with a function $f$ we obtain
\begin{equation} \label{eq: t_integral}
    \Theta \ast f(\boldsymbol x, t) = \int_{-\infty}^{t}f(\boldsymbol x,z)dz.
\end{equation}

We will use $\mathbbm{1}_{\Omega}$ to denote the indicator function of a domain $\Omega$:
\begin{equation}
    \mathbbm{1}_{\Omega}(\boldsymbol x) = \begin{cases}
    1 \quad &\text{if} \, \boldsymbol x \in \Omega \\
          0 \quad &\text{if} \, \boldsymbol x \notin \Omega. \\
    \end{cases}
\end{equation}

Let $T(\boldsymbol x,t)$ be a tempered distribution. When $T$ is supported in a half-space
\begin{equation}\label{eq:Ht}
H_{t_0} \coloneqq \{(\boldsymbol{x},t)|\,t\leq t_0\},
\end{equation}
its Fourier transform (which is a tempered distribution itself), 
\begin{equation}
    \hat{T}(\boldsymbol \omega, \sigma) = \mathcal{F}[T] (\boldsymbol \omega, \sigma) = T(e^{-i(\boldsymbol x \cdot \boldsymbol \omega + \sigma t)}),
\end{equation}
has analytic extension to $\Im(\sigma) > 0$ (See \citep{Bremmerman, Sharyn1999, Strichartz}).

We thus can simplify our considerations by working with the Fourier transform on the open set $\mathbb{H}_+=\{(\boldsymbol \omega, \sigma):\Im(\sigma) > 0\}$ and taking the limit $\Im(\sigma) \searrow0$ when needed. 
This applies to the distributions $D$ and $D^{\prime}$ introduced before. In particular (see, e.g.\citep{Stein}), $\widehat{D}$ on $\mathbb{H}_+$ can be computed as follows:
\begin{align*}
    \widehat{D}(\boldsymbol \omega,\sigma)
    &= 2^{-\frac{1}{2}} \int_{-\infty}^{0} \int_{\partial B(0,1)} \left ( e^{i r \boldsymbol \omega \cdot \boldsymbol \theta -i r \sigma) } \right ) d\boldsymbol \theta (-r)^{n-2} dr \\
    &= 2^{-\frac{1}{2}} \int_{\mathbb{R}^{n-1}} \left ( e^{i \sigma |\boldsymbol u|} e^{-i \boldsymbol \omega \cdot \boldsymbol u} \right ) d \boldsymbol u \\
    &=\alpha_n \frac{i\sigma}{(\sigma^2-|\boldsymbol \omega|^2)^{\frac{n}{2}}},
\end{align*}
where $d  \boldsymbol\theta$ is the surface measure on the sphere, $\boldsymbol u \in \mathbb{R}^{n-1}$, and
\begin{equation} \label{eq: alpha}
    \alpha_n  \coloneqq -(-1)^{-\frac{n}{2}} 2^{\frac{2n-1}{2}} \pi^{\frac{n-2}{2}} \Gamma \left ( \frac{n}{2} \right).
\end{equation}

Using the relation $\sqrt{2}tD^{1} = D$, we conclude that on $\mathbb{H}_+$ one has
\begin{equation} \label{eq: FFT_D'}
    \widehat{D^{1}}(\boldsymbol \omega,\sigma) = \frac{\beta_n}{(\sigma^2-|\boldsymbol \omega|^2)^{\frac{n-2}{2}}},
\end{equation}
where 
\begin{equation} \label{eq: beta}
    \beta_n  \coloneqq -(-1)^{-\frac{n}{2}} \frac{2^{n-1}}{2-n} \pi^{\frac{n-2}{2}} \Gamma \left ( \frac{n}{2} \right).
\end{equation}
Finally, one last identity that will be needed in Section \ref{S:proofs} for $\widehat{\Theta \ast f}$ when $f \in C_{0}^{\infty}(\mathbb{R}^n)$ and 
\begin{equation} \label{eq: vanishing_vertical_line_integral}
\int_{-\infty}^{\infty} f(\boldsymbol x, t)dt = 0 \qquad \forall \boldsymbol x \in \mathbb{R}^{n-1}.    
\end{equation}
In this case we compute
\begin{align} \label{eq: FourierVerticalLineIdentity}
    \widehat{\Theta \ast f} &= \int_{\mathbb{R}^{n-1}} \int_{-\infty}^{\infty} \int_{-\infty}^{t} f(\boldsymbol x, s)ds e^{-i \sigma t}dt e^{-i \boldsymbol \omega \cdot \boldsymbol x} d \boldsymbol x \nonumber \\
    &= \int_{\mathbb{R}^{n-1}} -\frac{1}{i \sigma}  \int_{-\infty}^{t} f(\boldsymbol x, s)ds e^{-i \sigma t}|_{t \rightarrow -\infty}^{t \rightarrow \infty} + \frac{1}{i \sigma} \int_{-\infty}^{\infty} f(\boldsymbol x, t) e^{-i \sigma t} dt e^{-i \boldsymbol \omega \cdot \boldsymbol x} d \boldsymbol x \nonumber \\
    &=\frac{1}{i \sigma} \int_{\mathbb{R}^{n-1}} f(\boldsymbol x, t) e^{-i \boldsymbol \omega \cdot \boldsymbol x-i \sigma t} d\boldsymbol x dt \nonumber \\
    &= \frac{1}{i \sigma} \hat{f} (\boldsymbol \omega, \sigma).
\end{align}
This identity holds for all $\sigma \in \mathbb{C} \backslash \{0\}$, and has analytic continuation to $\sigma = 0$. To see this observe that an equivalent expression for (\ref{eq: FourierVerticalLineIdentity}) is
\begin{equation}
    \int_{\mathbb{R}^{n-1}} \frac{1}{i \sigma} \int_{-\infty}^{\infty} f(\boldsymbol x, t) e^{-i \sigma t} dt e^{-i \boldsymbol \omega \cdot \boldsymbol x} d\boldsymbol x.
\end{equation}
Since $f$ is compactly supported $\hat{f}$ is entire according to the Paley-Wiener Theorem, and $f$ satisfies (\ref{eq: vanishing_vertical_line_integral}), therefore the inner integral vanishes when $\sigma=0$ and thus has power series representation with no zeroth order term:
\begin{align*}
    \int_{\mathbb{R}^{n-1}} \frac{1}{i \sigma} \int_{-\infty}^{\infty} f(\boldsymbol x, t) e^{-i \sigma t} dt e^{-i \boldsymbol \omega \cdot \boldsymbol x} d\boldsymbol x = \int_{\mathbb{R}^{n-1}} \frac{1}{i \sigma} \sum_{j=1}^{\infty} \xi_{j}(\boldsymbol x) \sigma^j e^{-i \boldsymbol \omega \cdot \boldsymbol x} d\boldsymbol x \\
    = -i \int_{\mathbb{R}^{n-1}} \sum_{j=1}^{\infty} \xi_{j}(\boldsymbol x) \sigma^{j-1} e^{-i \boldsymbol \omega \cdot \boldsymbol x} d\boldsymbol x \overset{\sigma \rightarrow 0}{\rightarrow} -i \int_{\mathbb{R}^{n-1}} \xi_{1}(\boldsymbol x) e^{-i \boldsymbol \omega \cdot \boldsymbol x} d\boldsymbol x.
\end{align*}
Each term $\xi_j$ is compactly supported, thus the above expression analytic.
\section{Range description for restricted CRT}\label{S:results}
As indicated in the Introduction, our first goal is to characterize the range of the restricted CRT $\mathcal{C}$ as a map from $C_{0}^{\infty}(\mathbb{R}^n)$ to $C^{\infty}(\mathbb{R}^n)$. As one might have expected, the answers differ in odd and even dimensions.
\begin{theorem}\label{thm: even} Let $n=2k$ be even (and thus $\boldsymbol x\in\mathbb{R}^{2k-1}$).
A function $g \in C^{\infty}(\mathbb{R}^{2k-1} \times \mathbb{R})$ is in the range of $\mathcal{C}$ on $C_{0}^{\infty}(\mathbb{R}^{2k-1} \times \mathbb{R})$ if and only if the following conditions are satisfied:
\begin{enumerate}[(i)]
    \item $\square^k g(\boldsymbol x,t)$ has compact support.
    \item $\int_{-\infty}^{\infty} \square^k g(\boldsymbol x,t)dt=0$ for every $\boldsymbol x \in \mathbb{R}^{2k-1}$.
    \item $\supp g \subseteq H_{t_0}$ (see (\ref{eq:Ht})) for some $t_0 \in \mathbb{R}$.
\end{enumerate}
\end{theorem}
An analogous result for odd dimensions is as follows:
\begin{theorem}\label{thm: odd}  Let $n=2k+1$ be odd (and thus $\boldsymbol x\in\mathbb{R}^{2k}$).
A function $g \in C^{\infty}(\mathbb{R}^{2k} \times \mathbb{R})$ is in the range of $\mathcal{C}$ on $C_{0}^{\infty}(\mathbb{R}^{2k} \times \mathbb{R})$ if and only if
\begin{enumerate}[(i)]
    \item $\square^{2k} \mathcal{C}^{1}[g](\boldsymbol x,t)$ has compact support.
    \item $\int_{-\infty}^{\infty} \square^{2k} \mathcal{C}^{1}[g](\boldsymbol x,t)dt=0$ for every $\boldsymbol x \in \mathbb{R}^{2k}$.
    \item $\supp g \subseteq H_{t_0}$ for some $t_0 \in \mathbb{R}$.
\end{enumerate}
\end{theorem}

As seen in these theorems, there is a close relationship between the CRT and the weighted CRT $\mathcal{C}^{1}$. In fact the range of $\mathcal{C}^{1}$ is very similar to the range of $\mathcal{C}$.

\begin{theorem} \label{thm: weighted_even} Let $n=2k$ be even (and thus $\boldsymbol x\in\mathbb{R}^{2k-1}$).
A function $g \in C^{\infty}(\mathbb{R}^{2k-1} \times \mathbb{R})$ is in the range of $\mathcal{C}^{1}$ on $C_{0}^{\infty}(\mathbb{R}^{2k-1} \times \mathbb{R})$ if and only if the following conditions are satisfied:
\begin{enumerate}[(i)]
    \item $\square^{k-1} g(\boldsymbol x,t)$ has compact support.
    \item $\supp g \subseteq H_{t_0}$ (see (\ref{eq:Ht})) for some $t_0 \in \mathbb{R}$.
\end{enumerate}
\end{theorem}

An analogous result for odd dimensions is as follows:
\begin{theorem}\label{thm: weighted_odd}  Let $n=2k+1$ be odd (and thus $\boldsymbol x\in\mathbb{R}^{2k}$).
A function $g \in C^{\infty}(\mathbb{R}^{2k} \times \mathbb{R})$ is in the range of $\mathcal{C}^{1}$ on $C_{0}^{\infty}(\mathbb{R}^{2k} \times \mathbb{R})$ if and only if
\begin{enumerate}[(i)]
    \item $\square^{2k-1} \mathcal{C}^{1}[g](\boldsymbol x,t)$ has compact support.
    \item $\supp g \subseteq H_{t_0}$ for some $t_0 \in \mathbb{R}$.
\end{enumerate}
\end{theorem}

\section{Proofs of Theorems}\label{S:proofs}
We start with the following auxiliary results:

\begin{lemma} \label{lm: dAlembert_Kernel} Let $g\in C^\infty(\R^n)$ be such that $\square^l g=0$ for some natural number $l$ and $\supp g\subset H_{t_0}$ for some $t_0 \in \R$, then $g\equiv 0$.
\end{lemma}
\begin{proof}
Let $g$ satisfy the conditions of the lemma and $l=1$. Then for any $s>t_0$, $g$ solves the Cauchy problem
\begin{equation} \label{eq: Cauchy_Homogeneous}
    \left \{\begin{array}{c}
     \square g = 0 \\
     g(\boldsymbol x,s) = 0\\
      g_t(\boldsymbol x,s) = 0\\
    \end{array} \right .
\end{equation}
(See Fig. \ref{fig: WaveEqnDomain}).
\begin{figure}[ht]
    \centering
    \includegraphics[width=0.6\textwidth]{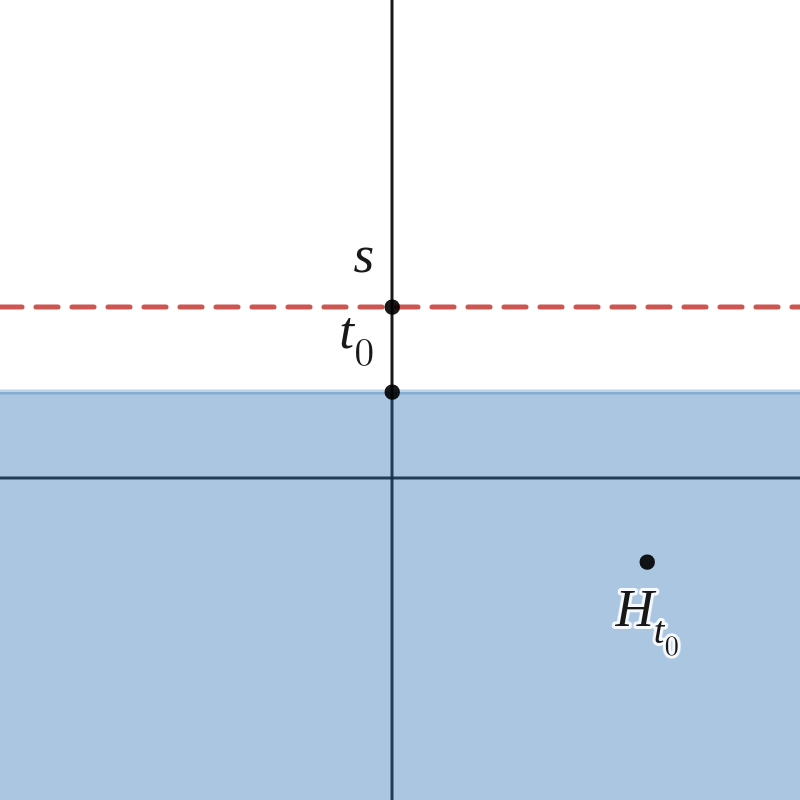}
    \caption{Domain of $g$ in Lemma \ref{lm: dAlembert_Kernel}. If $g$ vanishes outside a half-space, $g$ must vanish everywhere.}
    \label{fig: WaveEqnDomain}
\end{figure}
It is well known (see e.g. \citep{evans}) that the only solution to the Cauchy problem (\ref{eq: Cauchy_Homogeneous}) is $g \equiv 0$.

If $l>1$, then $v=(\square)^{l-1}g$ satisfies the same conditions of the lemma for $l=1$, and hence $v\equiv 0$. By induction we conclude that $g\equiv 0$.
\end{proof}
\begin{corollary} \label{co: iterated_kernel}
A smooth solution of  $\square^l g=f$ supported in $H_{t_0}$ for some $t_0\in\R$ (if such a solution exists) is unique.
\end{corollary}

\begin{lemma} \label{lm: bounded_even}
Let $g \in C^{\infty}(\mathbb{R}^{2k-1} \times \mathbb{R})$ satisfy the conditions of Theorem \ref{thm: even}, then $g$ and its derivatives are bounded.
\end{lemma}
\begin{proof}

To prove the claim that $g$ is bounded, due to Corollary \ref{co: iterated_kernel}, it will suffice to construct a solution of the PDE
\begin{equation} \label{eq: pde_lemma}
    \square^k h = f 
\end{equation}
that is supported in some $H_{t_0}$ and is indeed bounded\footnote{Here $f$ is \textbf{defined} as $\square^k g$.}.

We achieve this by using a fundamental solution $\Phi^{(k,2k)}$ for $\square^k$. 
Then the needed solution of (\ref{eq: pde_lemma}) will be given by the convolution $\Phi^{(k,2k)} \ast f$, which will be supported on a half-space. 

Fundamental solutions for $\square^\xi h$ for any complex number $\xi$ were studied extensively in \citep{Bollini}. In particular, for integer $k$ and dimension $2k$ the fundamental solution is given by
\begin{equation}\label{eq: fundamental_solution}
    \Phi^{(k,2k)} = (-1)^{k}\alpha_{2k}^{-1}\Theta \ast D.
\end{equation}
The fact that it is a fundamental solution is easily verified by making use of the Fourier transform for $\Im(\sigma) > 0$:
\begin{equation} \label{eq: fundamental_computation_even}
    \mathcal{F}[\square^k \Phi^{(k,2k)} \ast f] = \alpha_{2k}^{-1} (i \sigma)^{-1} (\sigma^2 - |\boldsymbol \omega|^2)^k \alpha_{2k} \frac{i \sigma}{(\sigma^2 - |\boldsymbol \omega|^2)^k} \hat{f} = \hat{f},
\end{equation}
where $\hat{f}$ denotes the Fourier transform of $f$.

Furthermore, the convolution $\Theta \ast D$ is supported outside a cone, due to the geometry of the supports of $\Theta$ and $D$ (see Fig. \ref{fig: VolCone}).
\begin{figure}[ht!]
    \centering
    \raisebox{-.5\height}{\includegraphics[width=0.3\textwidth,trim=0 4cm 0 0]{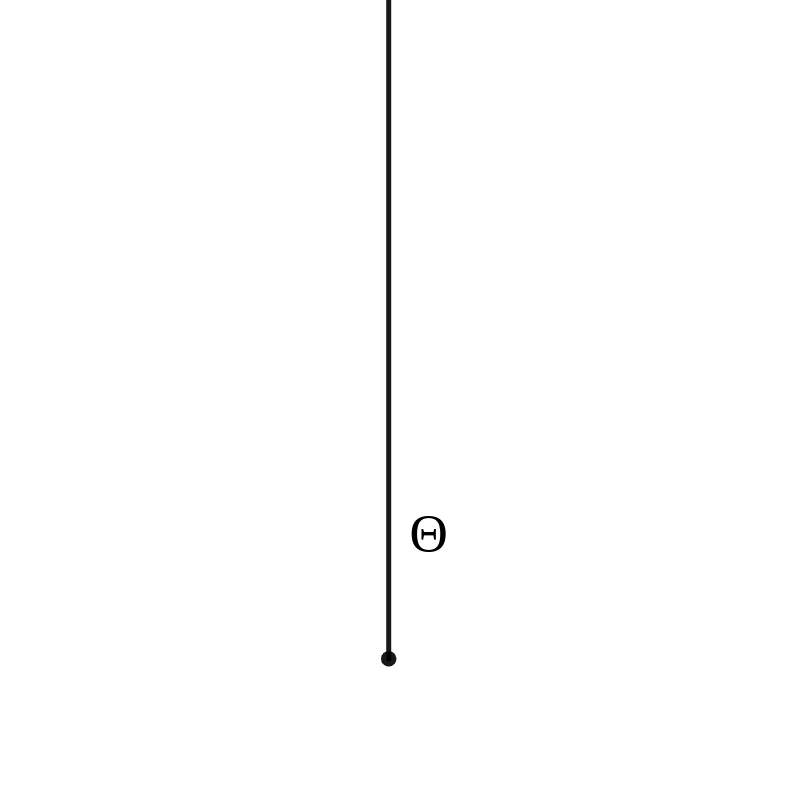}} $\ast$ \raisebox{-.5\height}{\includegraphics[width=0.3\textwidth,trim=0 0cm 0 0]{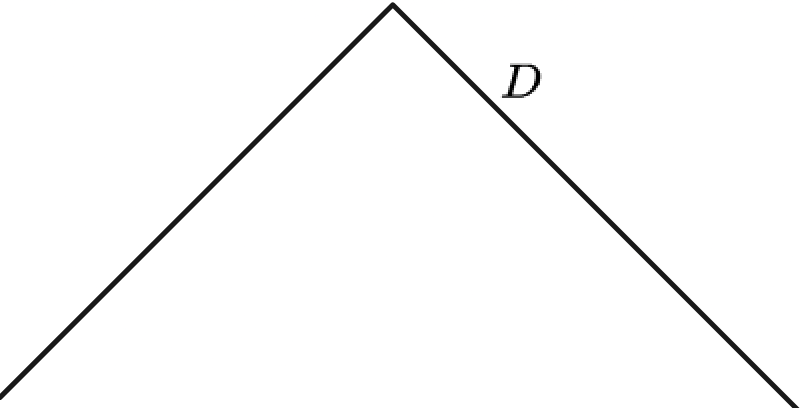}} $=$ \raisebox{-.5\height}{\includegraphics[width=0.3\textwidth,trim=0 2cm 0 0]{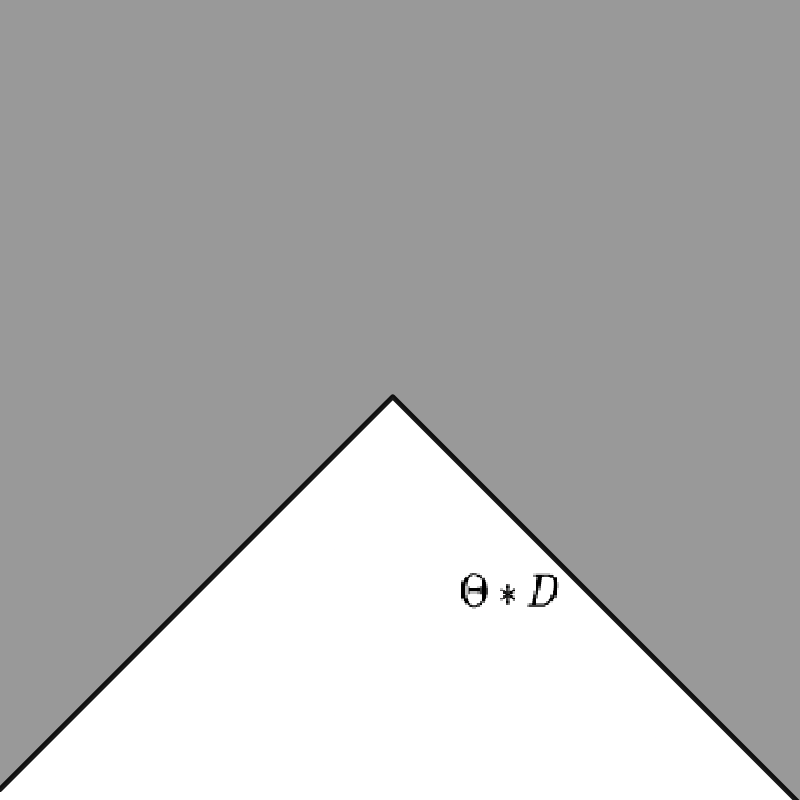}}
    \caption{The supports of $\Theta$ (left), $D$ (center), and their convolution (right).}
    \label{fig: VolCone}
\end{figure}

By construction, the function $h$ can be written $h=\Phi^{(k,2k)} \ast f$ (which is in fact closely related to the inversion formula for the conical Radon transform). Let us now show that $\Phi^{(k,2k)} \ast f$, and hence $g$, is bounded. Observe that
$$    |\Phi^{(k,2k)} \ast f(\boldsymbol x,t)| \leq |\alpha_{2k}|^{-1} \int_{|\boldsymbol x-\boldsymbol y| = t-s}  dS \int_{-\infty}^{s} |f(\boldsymbol y,z)|dz.
$$
Since the right hand side is a volume integral, we get the estimate
\begin{equation}\label{eq: bound}    
|\Phi^{(k,2k)} \ast f(\boldsymbol x,t)| \leq C \max (|f|) Vol(\supp (f)).
\end{equation}
This implies boundedness of $h=\Phi^{(k,2k)} \ast f$.

\end{proof}
\begin{remark}
By applying the same argument to $\partial^{m}\Phi^{(k,2k)} \ast f=\Phi^{(k,2k)} \ast \partial^m f$ where $m$ is a multiindex, one confirms that all derivatives are also bounded.
\end{remark}

\begin{lemma} \label{lm: bounded_odd}
Let $g \in C^{\infty}(\mathbb{R}^{2k} \times \mathbb{R})$ satisfy the conditions of Theorem \ref{thm: odd}, then $g$ and its derivatives are bounded.
\end{lemma}
\begin{proof}

Similarly to the previous lemma, we shall prove this by studying the solution of the equation
\begin{equation} \label{eq: ipde_lemma}
    \square^{2k} \mathcal{C}^{1}[h] = f \qquad f \in C_{0}^{\infty}(\mathbb{R}^{2k} \times \mathbb{R})
\end{equation}
under the constraint
\begin{equation} \label{eq: ipde_constraint}
    \int_{-\infty}^{\infty} \square^{2k} \mathcal{C}^{1}[h](\boldsymbol x,t)dt = 0 \qquad \forall \boldsymbol x \in \mathbb{R}^{2k}.
\end{equation}
Our approach here is similar to the one above: finding an appropriate fundamental solution of the equation
\begin{equation}\label{eq: ipde_fundamental}
    \square^{2k} \mathcal{C}^{1}[\Phi^{\left (\frac{2k+1}{2},2k+1 \right )}] = \delta
\end{equation}
and then using Corollary \ref{co: iterated_kernel}.

The relevant fundamental solution is
\begin{equation}\label{eq: fundamental_solution_odd}
    \Phi^{\left (\frac{2k+1}{2},2k+1 \right )} = \alpha_{2k+1}^{-1}\beta_{2k+1}^{-1}\Theta \ast D.
\end{equation}
To verify this, we make use of the Fourier transform and (\ref{eq: FFT_D'}) with $\Im(\sigma) > 0$ to get
\begin{align} \label{eq: FT_fundamental_solution}
    \mathcal{F}[\square^{2k} \mathcal{C}^{1}[ \Phi^{\left (\frac{2k+1}{2},2k+1 \right )} \ast f]] = \nonumber \\
    \alpha_{2k+1}^{-1} \beta_{2k+1}^{-1} \left (\frac{(\sigma^2 - |\boldsymbol \omega|^2)^{2k+1}}{i\sigma} \right) \alpha_{2k+1} \beta_{2k+1} \frac{i \sigma}{(\sigma^2 - |\boldsymbol \omega|^2)^{2k}} \hat{f} = \hat{f}.
\end{align}
Boundedness of $\Phi^{\left (\frac{2k+1}{2},2k+1 \right )} \ast f$ is proven exactly as in the previous case.
\end{proof}
\begin{remark}
Although we only require that $\supp g \subset H_{t_0}$, if $g$ is the CRT of some compactly supported function $f$ then it must be that $\supp g \subset \supp f + C(0,-\boldsymbol e_n, \pi/4)$. One can easily verify that the conditions imposed on $g$ in Theorems \ref{thm: even} and \ref{thm: odd} guarantee this.
\end{remark}

\begin{lemma} \label{lm: weighted_bounded_even}
Let $g \in C^{\infty}(\mathbb{R}^{2k-1} \times \mathbb{R})$ satisfy the conditions of Theorem \ref{thm: weighted_even}, then $g$ and its derivatives are bounded.
\end{lemma}
\begin{proof}

Again we seek a solution of the PDE
\begin{equation} \label{eq: weighted_pde_lemma}
    \square^{k-1} h = f 
\end{equation}
that is supported in some $H_{t_0}$.

The relevant fundamental solution $\Phi^{(k,2k)}$ for $\square^{k-1}$ is given by
\begin{equation}\label{eq: weighted_fundamental_solution}
    \Phi^{(k-1,2k)} = (-1)^{k-1}\beta_{2k}^{-1} \ast D^{1}.
\end{equation}
The fact that it is a fundamental solution is easily verified by making use of the Fourier transform for $\Im(\sigma) > 0$:
\begin{equation} \label{eq: weighted_fundamental_computation_even}
    \mathcal{F}[\square^{k-1} \Phi^{(k-1,2k)} \ast f] = \beta_{2k}^{-1} (\sigma^2 - |\boldsymbol \omega|^2)^{k-1}  \frac{\beta_{2k}}{(\sigma^2 - |\boldsymbol \omega|^2)^{k-1}} \hat{f} = \hat{f},
\end{equation}
where $\hat{f}$ denotes the Fourier transform of $f$.

By construction, the function $h$ can be written $h=\Phi^{(k-1,2k)} \ast f$. Let us now show that $\Phi^{(k-1,2k)} \ast f$, and hence $h$ is bounded. Observe that
$$    |\Phi^{(k-1,2k)} \ast f(\boldsymbol x,t)| \leq |\beta_{2k}|^{-1} \int_{|\boldsymbol x-\boldsymbol y| = t-s} f(\boldsymbol x,t) dS \leq |\beta_{2k}|^{-1} \int_{\mathbb{R}} \int_{\mathbb{R}^{n-1}} f(\boldsymbol x,t) d\boldsymbol x dt.
$$
Since the right hand side is a volume integral, we get the estimate
\begin{equation}\label{eq: weighted_bound}    
|\Phi^{(k-1,2k)} \ast f(\boldsymbol x,t)| \leq C \max (|f|) Vol(\supp (f)).
\end{equation}
This implies boundedness of $h=\Phi^{(k-1,2k)} \ast f$.

\end{proof}
\begin{remark}
By applying the same argument to $\partial^{m}\Phi^{(k-1,2k)} \ast f=\Phi^{(k-1,2k)} \ast \partial^m f$ where $m$ is a multiindex, one confirms that all derivatives are also bounded.
\end{remark}

\begin{lemma} \label{lm: weighted_bounded_odd}
Let $g \in C^{\infty}(\mathbb{R}^{2k} \times \mathbb{R})$ satisfy the conditions of Theorem \ref{thm: weighted_odd}, then $g$ and its derivatives are bounded.
\end{lemma}
\begin{proof}

Similarly to the previous lemma, we shall prove this by studying the solution of the equation
\begin{equation} \label{eq: weighted_ipde_lemma}
    \square^{2k-1} \mathcal{C}^{1}[h] = f \qquad f \in C_{0}^{\infty}(\mathbb{R}^{2k} \times \mathbb{R})
\end{equation}

The relevant fundamental solution is
\begin{equation}\label{eq: weighted_fundamental_solution_odd}
    \Phi^{\left (\frac{2k-1}{2},2k+1 \right )} = (-1)^{2k-1} \beta_{2k+1}^{-2} D^{1}.
\end{equation}
To verify this, we make use of the Fourier transform and (\ref{eq: FFT_D'}) with $\Im(\sigma) > 0$ to get
\begin{align} \label{eq: weighted_FT_fundamental_solution}
    \mathcal{F}[\square^{2k-2} \mathcal{C}^{1}[ \Phi^{\left (\frac{2k-1}{2},2k+1 \right )} \ast f]] = \nonumber \\
    \beta_{2k+1}^{-2} \left ((\sigma^2 - |\boldsymbol \omega|^2)^{2k-1} \right)   \frac{\beta_{2k+1}}{(\sigma^2 - |\boldsymbol \omega|^2)^{\frac{2k+1}{2}}} \frac{\beta_{2k+1}}{(\sigma^2 - |\boldsymbol \omega|^2)^{\frac{2k+1}{2}}} \hat{f} = \hat{f}.
\end{align}
Boundedness of $\Phi^{\left (\frac{2k-1}{2},2k+1 \right )} \ast f$ is proven exactly as in the previous case.
\end{proof}

The upshot of these lemmas is that $g$ is well-behaved enough for it and its derivatives to be identified with tempered distributions supported in a half-space, in which case we can utilize computations with Fourier transforms freely.

\subsection{Proof of Theorem \ref{thm: even}}

\begin{proof}
Let us start proving the necessity of the conditions. Suppose $g= \mathcal{C}[f]$ for some $f \in C_{0}^{\infty}(\mathbb{R}^{2k-1} \times \mathbb{R})$ such that $\supp{f} \subset H_{t_0}$. It follows from the definition of CRT that $g$ is smooth, bounded and $\supp g \subset H_{t_0}$ (in particular, the condition (iii) of the theorem holds)\footnote{Since $g=D \ast f$, the support of $g$ belongs to the sum of the supports of $f$ and the distribution $D$. Because  $\supp f \subseteq H_{t_0}$, $\supp g \subseteq H_{t_0}+H_0=H_{t_0}$.}. Therefore, for $\Im(\sigma) > 0$ one has
\begin{align*}
    \widehat{\square^k g}(\boldsymbol \omega,\sigma) &= (-1)^{k} (\sigma^2 - |\boldsymbol \omega|^2)^k \alpha_{2k} \frac{i \sigma}{(\sigma^2 - |\boldsymbol \omega|^2)^k} \hat{f}(\boldsymbol \omega,\sigma) \\
    &= (-1)^{k} \alpha_{2k} i \sigma \hat{f}(\boldsymbol \omega,\sigma).
\end{align*}
In the limit $\Im(\sigma) \rightarrow 0^+$, we obtain
\begin{equation}
    \square^k g(\boldsymbol x,t) = (-1)^{k} \alpha_{2k} f_t(\boldsymbol x,t).
\end{equation}
Conditions (i) and (ii) follow immediately from the previous identity, which finishes the proof of necessity of the conditions.

To prove the sufficiency of the conditions we will make use of the inverse CRT. Suppose $g$ is smooth and satisfies the conditions specified in Theorem \ref{thm: even}. Define
\begin{equation}
    f( \boldsymbol x,t): = (-1)^{k}\alpha_{2k}^{-1} \int_{-\infty}^{t} \square^k g(\boldsymbol x,z)dz.
\end{equation}
Then $f$ is smooth with compact support, as $g$ is smooth and satisfies conditions (i) and (ii). Moreover, we have for $\Im(\sigma) > 0$

\begin{align*}
    \widehat{\mathcal{C}[f]}(\boldsymbol \omega,\sigma) &= \frac{\alpha_{2k}}{\alpha_{2k}}  \frac{i \sigma}{(\sigma^2 - |\boldsymbol \omega|^2)^k} \left ( \frac{1}{i \sigma} (\sigma^2 - |\boldsymbol \omega|^2)^k \hat{g}  \right )=\hat{g}.
\end{align*}
In the limit $\Im(\sigma) \rightarrow 0^+$ we find that in fact $\mathcal{C}[f]=g$ and thus $g$ is in the range\footnote{Note that this computation is justified by Lemma \ref{lm: bounded_even}.} of $\mathcal{C}$.
\end{proof}

\subsection{Proof of Theorem \ref{thm: odd}}
\begin{proof}
We proceed in the same vein as in the previous proof. Starting with proving necessity, suppose $g= \mathcal{C}[f]$ for some $f \in C_{0}^{\infty}(\mathbb{R}^{2k} \times \mathbb{R})$, and $\supp{f} \subset H_{t_0}$. Then, as in the previous theorem,  $g$ is smooth and bounded, and $\supp g\subset H_{t_0}$. Thus, for $\Im(\sigma) > 0$ we get
\begin{align*}
    \widehat{\square^{2k} \mathcal{C}^{1}[ g]}(\boldsymbol \omega,\sigma) &= \frac{\beta_{2k+1}}{(\sigma^2 - |\boldsymbol \omega|^2)^{k-1/2}} (\sigma^2 - |\boldsymbol \omega|^2)^{2k} \alpha_{2k+1} \frac{ i \sigma}{(\sigma^2 - |\boldsymbol \omega|2)^{k+1/2}} \hat{f}(\boldsymbol \omega,\sigma) \\
    &= \alpha_{2k+1} \beta_{2k+1} i \sigma \hat{f}(\boldsymbol \omega,\sigma).
\end{align*}
Sending $\Im(\sigma) \rightarrow 0^+$, we obtain
\begin{equation}
    \square^{2k} \mathcal{C}^{1}[ g](\boldsymbol x,t) = \alpha_{2k+1} \beta_{2k+1}  f_t(\boldsymbol x,t).
\end{equation}
This implies the conditions (i) and (ii), while (iii) has already been established. 

Let us turn to proving sufficiency of the conditions. Suppose $g$ satisfies the conditions of Theorem \ref{thm: odd}. We define
\begin{equation}
    f(\boldsymbol x,t): = \alpha_{2k+1}^{-1} \beta_{2k+1}^{-1} \int_{-\infty}^{t} \square^{2k} \mathcal{C}^{1}[ g](\boldsymbol x,z)dz.
\end{equation}
Then $f$ is smooth with compact support, as $g$ is smooth and satisfies conditions (i) and (ii). Moreover, we have for $\Im(\sigma) >0$

\begin{align*}
    \widehat{\mathcal{C}[f]}(\boldsymbol \omega,\sigma) = \frac{\alpha_{2k+1}}{\alpha_{2k+1}\beta_{2k+1}} \frac{i\sigma}{(\sigma^2 - |\boldsymbol \omega|^2)^{k+1/2}} \left ( \frac{\beta_{2k+1} (\sigma^2 - |\boldsymbol \omega|^2)^{2k}}{ i \sigma (\sigma^2 - |\boldsymbol \omega|^2)^{k-1/2}} \right ) \hat{g}=\hat{g}.
\end{align*}
Then in the limit $\Im(\sigma) \rightarrow 0^+$ we find that $\mathcal{C}[f]=g$ and thus $g$ is in the range of $\mathcal{C}$.
\end{proof}

\subsection{Proof of Theorem \ref{thm: weighted_even}}

\begin{proof}
Once again, we begin with the proof of necessity. Suppose $g= \mathcal{C}^{1}[f]$ for some $f \in C_{0}^{\infty}(\mathbb{R}^{2k-1} \times \mathbb{R})$ such that $\supp{f} \subset H_{t_0}$. It follows from the definition of weighted CRT that $g$ is smooth, bounded and $\supp g \subset H_{t_0}$ (in particular, the condition (ii) of the theorem holds). Therefore, for $\Im(\sigma) > 0$ one has
\begin{align*}
    \widehat{\square^{k-1} g}(\boldsymbol \omega,\sigma) &= (-1)^{k-1} (\sigma^2 - |\boldsymbol \omega|^2)^{k-1} \frac{\beta_{2k}}{(\sigma^2 - |\boldsymbol \omega|^2)^{k-1}} \hat{f}(\boldsymbol \omega,\sigma) \\
    &= (-1)^{k-1} \beta_{2k} \hat{f}(\boldsymbol \omega,\sigma).
\end{align*}
In the limit $\Im(\sigma) \rightarrow 0^+$, we obtain
\begin{equation}
    \square^{k-1} g(\boldsymbol x,t) = (-1)^{k-1} \beta_{2k} f(\boldsymbol x,t).
\end{equation}
Condition (i) of the theorem follows immediately from the previous identity, which finishes the proof of necessity of the conditions.

Now to prove sufficiency, suppose $g$ is smooth and satisfies the conditions specified in Theorem \ref{thm: weighted_even}. Define
\begin{equation}
    f( \boldsymbol x,t): = (-1)^{k-1}\beta_{2k}^{-1} \square^{k-1} g(\boldsymbol x,t).
\end{equation}
Then $f$ is smooth with compact support, as $g$ is smooth and satisfies condition (i). Moreover, we have for $\Im(\sigma) > 0$

\begin{align*}
    \widehat{\mathcal{C}^{1}[f]}(\boldsymbol \omega,\sigma) &= \frac{\beta_{2k}}{\beta_{2k}}  \frac{1}{(\sigma^2 - |\boldsymbol \omega|^2)^{k-1}} \left ( (\sigma^2 - |\boldsymbol \omega|^2)^{k-1} \hat{g}  \right )=\hat{g}.
\end{align*}
In the limit $\Im(\sigma) \rightarrow 0^+$ we find that in fact $\mathcal{C}[f]=g$ and thus $g$ is in the range of $\mathcal{C}$.
\end{proof}

\subsection{Proof of Theorem \ref{thm: weighted_odd}}
\begin{proof}
Suppose $g= \mathcal{C}[f]$ for some $f \in C_{0}^{\infty}(\mathbb{R}^{2k} \times \mathbb{R})$, and $\supp{f} \subset H_{t_0}$. Then,  $g$ is smooth and bounded, and $\supp g\subset H_{t_0}$. Thus, for $\Im(\sigma) > 0$ we get
\begin{align*}
    \widehat{\square^{2k-1} \mathcal{C}^{1}[ g]}(\boldsymbol \omega,\sigma) &= \frac{\beta_{2k+1}}{(\sigma^2 - |\boldsymbol \omega|^2)^{k-1/2}} (\sigma^2 - |\boldsymbol \omega|^2)^{2k-1} \frac{\beta_{2k+1} }{(\sigma^2 - |\boldsymbol \omega|2)^{k-1/2}} \hat{f}(\boldsymbol \omega,\sigma) \\
    &=\beta_{2k+1}^{2} \hat{f}(\boldsymbol \omega,\sigma).
\end{align*}
Sending $\Im(\sigma) \rightarrow 0^+$, we obtain
\begin{equation}
    \square^{2k-1} \mathcal{C}^{1}[ g](\boldsymbol x,t) =  \beta_{2k+1}^{2}  f(\boldsymbol x,t).
\end{equation}
This implies the condition (i). 

Now suppose $g$ satisfies the conditions of Theorem \ref{thm: weighted_odd}. We define
\begin{equation}
    f(\boldsymbol x,t): = (-1)^{2k-1} \beta_{2k+1}^{-2} \square^{2k-1} \mathcal{C}^{1}[ g](\boldsymbol x,z)dz.
\end{equation}
Then $f$ is smooth with compact support, as $g$ is smooth and satisfies condition (i). Moreover, we have for $\Im(\sigma) >0$

\begin{align*}
    \widehat{\mathcal{C}^{1}[f]}(\boldsymbol \omega,\sigma) = \frac{\beta_{2k+1}}{\beta_{2k+1}^2} \frac{1}{(\sigma^2 - |\boldsymbol \omega|^2)^{k-1/2}} \left ( \frac{\beta_{2k+1} (\sigma^2 - |\boldsymbol \omega|^2)^{2k-1}}{ (\sigma^2 - |\boldsymbol \omega|^2)^{k-1/2}} \right ) \hat{g} =\hat{g}.
\end{align*}
Then in the limit $\Im(\sigma) \rightarrow 0^+$ we find that $\mathcal{C}[f]=g$ and thus $g$ is in the range of $\mathcal{C}$.
\end{proof}
\section{Arbitrary angle of the cone}\label{S:angle}
For simplicity we have so far restricted our attention to cones with a right angle opening and central axis aligned with coordinate $t$ of our chosen coordinate system $(\boldsymbol x,t)$. The axis' direction does not restrict the generality. Let us tackle an arbitrary half-opening angle $\phi$. The mapping $(x,t)\mapsto ((\cot \phi) \boldsymbol x,t)$ transforms the $\pi/4$ half-opening angle to $\phi$.  

This suggests consideration of the modified d'Alembertian 
\begin{equation} \label{eq: dAlembertian_COB}
    \square_{\phi}  \coloneqq \frac{\partial^2}{\partial t^{2}} - (\tan^2 \phi) \Delta_{\boldsymbol x}
\end{equation}
Arguing along the same line as the previous sections will prove the results stated below\footnote{One could also obtain them by changing variables.} for the CRT $\mC_\phi$ with half-opening angle $\phi$.
\begin{theorem}\label{thm: even_COB} Let $n=2k$ be even. A function $g\in C_0^\infty(\R^{2k-1}\times \R)$ is in the range of $\mC_\phi$ acting on $C_0^\infty(\R^n)$ if and only if the following conditions are satisfied: 
\begin{enumerate}[(i)]
    \item $\square_{\phi}^{k} g(\boldsymbol x,t)$ has compact support.
    \item $\int_{-\infty}^{\infty} \square_{\phi}^{k} g(\boldsymbol x,t)dt=0$ for every $\boldsymbol x \in \mathbb{R}^{2k-1}$.
    \item $\supp g \subseteq H_{t_0}$ for some $t_0 \in \mathbb{R}$.
\end{enumerate}
\end{theorem}
An analogous result for odd dimensions is as follows:
\begin{theorem}\label{thm: odd_COB} Let $n=2k+1$ be odd.
A function $g \in C^{\infty}(\mathbb{R}^{2k} \times \mathbb{R})$ is in the range of $\mathcal{C}_{\phi}$ on $C_{0}^{\infty}(\mathbb{R}^{2k} \times \mathbb{R})$ if and only if the following conditions are satisfied:
\begin{enumerate}[(i)]
    \item $\square_{\phi}^{2k} \mathcal{C}_{\phi}^{1}[g](\boldsymbol x,t)$ has compact support.
    \item $\int_{-\infty}^{\infty} \square_{\phi}^{2k} \mathcal{C}_{\phi}^{1}[g](\boldsymbol x,t)dt=0$ for every $\boldsymbol x \in \mathbb{R}^{2k}$.
    \item $\supp g \subseteq H_{t_0}$ for some $t_0 \in \mathbb{R}$.
\end{enumerate}
where $\mathcal{C}_{\phi}^{1}$ is the weighted CRT corresponding to the opening angle $\phi$:
\begin{equation}
    \mathcal{C}_{\phi}^{1}[f](\boldsymbol x, t) =  D_{\phi}^{1} \ast f= \left [ |((\cot \phi) \boldsymbol x, t)|^{-1} \delta(t - \cot \phi |x|) \right ] \ast f
\end{equation}
\end{theorem}

\begin{theorem} \label{thm: weighted_even_COB} Let $n=2k$ be even (and thus $\boldsymbol x\in\mathbb{R}^{2k-1}$).
A function $g \in C^{\infty}(\mathbb{R}^{2k-1} \times \mathbb{R})$ is in the range of $\mathcal{C}_{\phi}^{1}$ on $C_{0}^{\infty}(\mathbb{R}^{2k-1} \times \mathbb{R})$ if and only if the following conditions are satisfied:
\begin{enumerate}[(i)]
    \item $\square_{\phi}^{k-1} g(\boldsymbol x,t)$ has compact support.
    \item $\supp g \subseteq H_{t_0}$ (see (\ref{eq:Ht})) for some $t_0 \in \mathbb{R}$.
\end{enumerate}
\end{theorem}

An analogous result for odd dimensions is as follows:
\begin{theorem}\label{thm: weighted_odd_COB}  Let $n=2k+1$ be odd (and thus $\boldsymbol x\in\mathbb{R}^{2k}$).
A function $g \in C^{\infty}(\mathbb{R}^{2k} \times \mathbb{R})$ is in the range of $\mathcal{C}_{\phi}^{1}$ on $C_{0}^{\infty}(\mathbb{R}^{2k} \times \mathbb{R})$ if and only if
\begin{enumerate}[(i)]
    \item $\square_{\phi}^{2k-1} \mathcal{C}_{\phi}^{1}[g](\boldsymbol x,t)$ has compact support.
    \item $\supp g \subseteq H_{t_0}$ for some $t_0 \in \mathbb{R}$.
\end{enumerate}
\end{theorem}
\subsection{Proof of Theorems for an arbitrary angle of the cone}
First note that Lemma \ref{lm: dAlembert_Kernel} holds for any wave speed, in particular the result holds for $\square_\phi$. Next, observe that the critical computations in the proofs of Theorems \ref{thm: even} and \ref{thm: odd} are (\ref{eq: fundamental_computation_even}) and (\ref{eq: FT_fundamental_solution}), which are still valid due to the respective fundamental solutions being tempered distributions supported on a half-space. It will therefore suffice to produce tempered distributions which are fundamental solutions to the PDEs:
\begin{equation} \label{eq: PDE_COB_even}
    \square_{\phi}^k \Phi_{\phi}^{(k,2k)} = \delta
\end{equation}
and
\begin{equation} \label{eq: PDE_COB_odd}
    \square_{\phi}^{2k} \mathcal{C}_{\phi}^{1}[\Phi_{\phi}^{\left (\frac{2k+1}{2},2k+1 \right )}] = \delta.
\end{equation}
We can repeat the computation at the end of Section \ref{S:notation} to determine the Fourier transforms of $D_\phi$ and $D_{\phi}^{1}$ on $\mathbb{H}_+$:
\begin{align*}
    \widehat{D_\phi}(\boldsymbol \omega,\sigma)
    &= \sin^{n-2} \phi \sec^{n-1} \phi \int_{-\infty}^{0} \int_{\partial B(0,1)} \left ( e^{ir \tan \phi \boldsymbol \omega \cdot \boldsymbol \theta -i r \sigma) } \right ) d\boldsymbol \theta (-r)^{n-2} dr \\
    &=\alpha_{n,\phi} \frac{i\sigma}{(\sigma^2-\tan^2 \phi|\boldsymbol \omega|^2)^{\frac{n}{2}}},
\end{align*}
where now,
\begin{equation} \label{eq: alpha_phi}
    \alpha_{n,\phi}  \coloneqq -\sin^{n-2} \phi \sec^{n-1} \phi (-1)^{-\frac{n}{2}} 2^{n-1} \pi^{\frac{n-2}{2}} \Gamma \left ( \frac{n}{2} \right).
\end{equation}
Using the relation $t \sec \phi D_{\phi}^{1} = D_\phi$, we conclude that on $\mathbb{H}_+$ one has
\begin{equation} \label{eq: FFT_D'_phi}
    \widehat{D^{1}}(\boldsymbol \omega,\sigma) = \frac{\beta_{n,\phi}}{(\sigma^2-\tan^2 \phi |\boldsymbol \omega|^2)^{\frac{n-2}{2}}},
\end{equation}
where now 
\begin{equation} \label{eq: beta_phi}
    \beta_{n,\phi}  \coloneqq - \sin^{n-2} \phi \sec^{n-2} \phi (-1)^{-\frac{n}{2}} \frac{2^{n-1}}{2-n} \pi^{\frac{n-2}{2}} \Gamma \left ( \frac{n}{2} \right).
\end{equation}
The fundamental solution for (\ref{eq: PDE_COB_even}) is
\begin{equation}\label{eq: fundamental_solution_COB_even}
    \Phi_{\phi}^{(k,2k)} = (-1)^{k}\alpha_{2k,\phi}^{-1}\Theta \ast D_\phi
\end{equation}
and the fundamental solution for (\ref{eq: PDE_COB_odd}) is
\begin{equation}\label{eq: fundamental_solution_COB_odd}
    \Phi_{\phi}^{\left (\frac{2k+1}{2},2k+1 \right )} = \alpha_{2k+1,\phi}^{-1}\beta_{2k+1,\phi}^{-1}\Theta \ast D_\phi.
\end{equation}
The remainder of the proofs is identical to the proofs in Section \ref{S:proofs}.
\section{A range description of the full CRT} \label{S: Full_Range}

We now turn our attention to the full CRT, i.e. the one that uses \emph{all} cones. Before we discuss its range, we recall the divergent beam and spherical mean transforms. The reason is that the CRT can be factored into the composition of the former two \citep{Fatma}. This will induce certain additional characteristics on the range of the CRT. These new features, along with Theorems \ref{thm: even_COB} and \ref{thm: odd_COB} will lead to a description of the range of the full CRT.

\subsection{The Divergent Beam Transform} \label{sec: divergent_beam}
\begin{figure}[ht!]
    \centering
    \includegraphics[width=0.4\textwidth]{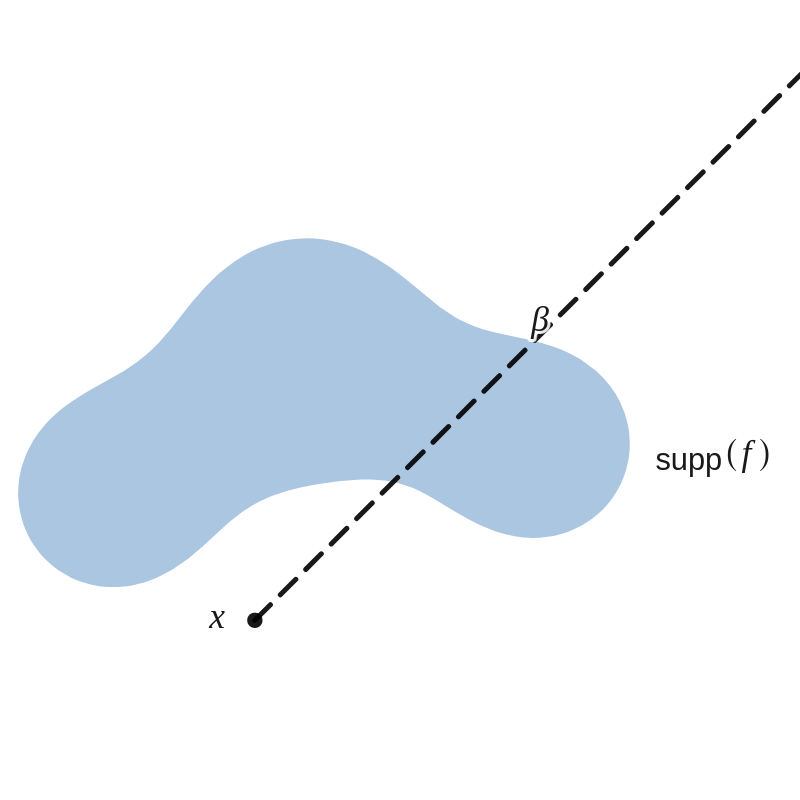}
    \caption{Schematic representation of the divergent beam transform.}
    \label{fig: DivergentBeam}
\end{figure}
The divergent beam transform $P[f](\boldsymbol x, \boldsymbol \beta)$ of $f\in C_0^{\infty}(\R^n)$ is the integral of $f$ over the ray originating at $\boldsymbol x$ and emanating in the direction $\boldsymbol \beta \in \mathbb{S}^{n-1}$ \citep{Natt} (see Fig. \ref{fig: DivergentBeam}):
\begin{equation} \label{eq: divergent_beam_transform}
    P[f](\boldsymbol x, \boldsymbol \beta)= \int_{0}^{\infty} f(\boldsymbol x + t \boldsymbol \beta) dt.
\end{equation}
It is evident from the definition that for $f\in C_0^{\infty}(\R^n)$, $P[f]\in C^{\infty}(\R^n \times \mathbb{S}^{n-1})$.
\subsection{The Spherical Mean Transform} \label{sec: spherical_mean}
\begin{figure}[ht!]
    \centering
    \includegraphics[width=0.4\textwidth]{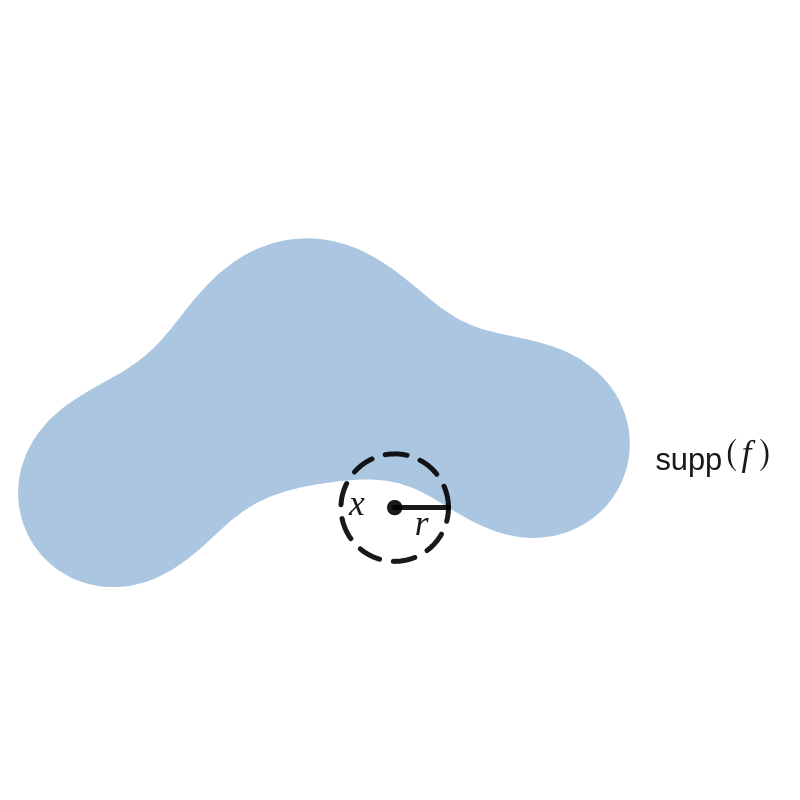}
    \caption{Schematic representation of the spherical mean transform.}
    \label{fig: SphericalMean}
\end{figure}
The spherical mean transform $M_S[f](\boldsymbol x,r)$ of $f\in C^{\infty}(\R^n)$ is the average of $f$ over the sphere with center $\boldsymbol x$ and radius $r\in \R_+$ (see \citep{Fritz} and Fig. \ref{fig: SphericalMean}):
\begin{equation} \label{eq: spherical_mean_transform}
    M_S[f](\boldsymbol x, r)= \frac{1}{\omega_{n-1}} \int_{\partial B(\boldsymbol x, r)} f(\boldsymbol y) d\boldsymbol \theta (\boldsymbol y)
\end{equation}
where $\omega_n$ is the surface area of the $n-$sphere. The spherical mean transform is extremely interesting in its own right, and several works are devoted to its study (see e.g. \citep{Agranovsky2009,Helgason,Fritz,Rubin}). A key property of the spherical mean transform is that its range is contained in the kernel of the Euler--Poisson--Darboux operator \citep{Fritz,Rubin}:
\begin{equation} \label{eq: EPD}
    \mathcal{E}  \coloneqq \frac{\partial ^2}{\partial r^2} +\frac{n-1}{r} \frac{\partial}{\partial r} - \Delta_{\boldsymbol x}
\end{equation}
\begin{figure}[ht!]
    \centering
    \includegraphics[width=0.4\textwidth]{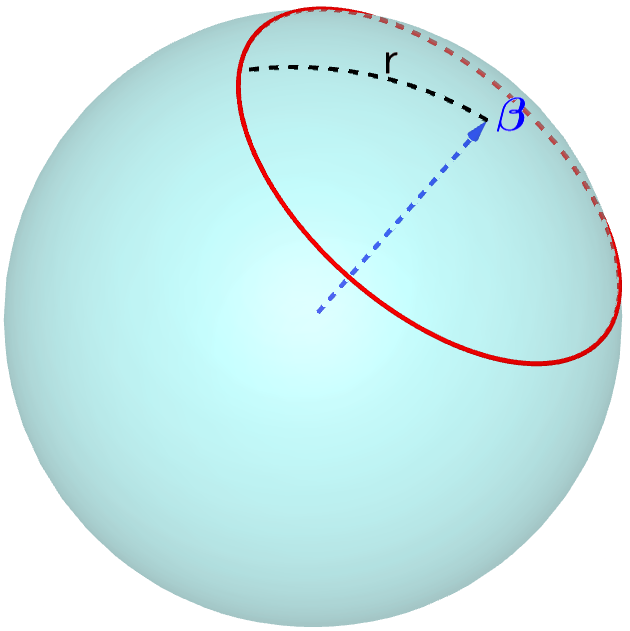}
    \caption{Schematic representation of the spherical mean transform on the sphere.}
    \label{fig: SphericalMeanOnSphere}
\end{figure}
In fact, these results can be generalized to the spherical mean transform over a wide class of manifolds (namely, two-point homogeneous spaces) \citep{Helgason}. The case of the sphere will be of particular interest to us. 

The spherical mean transform $M[f](\boldsymbol \beta,r)$ of $f\in C^{\infty}(\mathbb{S}^{n-1})$ is the average of $f$ over the $(n-2)$-sphere with center $\boldsymbol \beta \in \mathbb{S}^{n-2}$ and geodesic radius $r\in (0,\pi/2)$ (see Fig. \ref{fig: SphericalMeanOnSphere}):
\begin{equation} \label{eq: spherical_mean_transform_on_sphere}
    M[f](\boldsymbol \beta, r)= \frac{1}{A_{n-2}(r)} \int_{\partial B(\boldsymbol \beta, r)} f(\boldsymbol y) d\boldsymbol \theta (\boldsymbol y)
\end{equation}
where  $A_n(r)$ is the surface area of the $n-$sphere with geodesic radius $r$. From this point forward we will refer to (\ref{eq: spherical_mean_transform_on_sphere}) simply as the spherical mean transform. Two key properties of this spherical mean transform which will be critical in our discussion of the range of the CRT are the following injectivity result and range description.
\begin{theorem}[Injectivity (Helgason \citep{Helgason})] \label{thm: injectivity}
For any fixed $r>0$ the transform $M[f](\boldsymbol \beta,r)$ is injective for $f \in L^1(\mathbb{S}^{n-1})$.
\end{theorem}
\begin{theorem}[Range (Helgason\protect\footnotemark \citep{Helgason})] \label{thm: range}
The range of the transform $M[f](\boldsymbol \beta,r)$ over $C^2(\mathbb{S}^{n-2})$ is the kernel of the Euler--Poisson--Darboux type operator 
\begin{equation} \label{eq: EPD_sphere}
    \mathcal{E}_S \coloneqq \frac{\partial^2 }{\partial r^2} + \frac{A'(r)}{A(r)} \frac{\partial }{\partial r} - \Delta_{S}
\end{equation}
where $\Delta_S$ is the Laplace-Beltrami operator on the sphere $\mathbb{S}^{n-1}$ and $A'(r)$ is the derivative of $A(r)$.
\end{theorem}
\footnotetext{Helgason in fact proved in \citep{Helgason} these two results for the general case of the spherical mean transforms on compact two-point homogeneous spaces.}
\subsection{Factoring the CRT} \label{sec: factor}
As before, the CRT can be factored into the composition of the divergent beam and spherical mean transforms. This can be seen through the following computation:
\begin{align*}
    \mathcal{C}[f](\boldsymbol x, \phi, \boldsymbol \beta) &= \sin \phi \int_{\mathbb{R}^n} f(\boldsymbol y) \delta ((\boldsymbol y - \boldsymbol x)\cdot \boldsymbol \beta - |\boldsymbol y - \boldsymbol x| \cos \phi)d\boldsymbol y \\
    & =  \sin \phi \int_{\mathbb{R}^n} f(\boldsymbol x -\boldsymbol y) \delta (\boldsymbol y\cdot \boldsymbol \beta - |\boldsymbol y| \cos \phi)d\boldsymbol y \\
    & = \sin \phi \int_{\mathbb{S}^{n-1}} P[f](\boldsymbol x, \boldsymbol y) \delta (\boldsymbol y\cdot \boldsymbol \beta -\cos \phi)d\theta(y) \\
    & = \sin \phi \int_{\mathbb{S}^{n-2}} P[f](\boldsymbol x,\cos \phi \boldsymbol \beta + \sin \phi \boldsymbol y) d\theta(y) \\
    & = \sin \phi A(\phi) M[P[f](\boldsymbol x, \boldsymbol \alpha)](\boldsymbol \beta, \phi)),
\end{align*}
where the spherical mean transform $M[P(\boldsymbol x,\boldsymbol \alpha)]$ is taken with respect to $\alpha \in \mathbb{S}^{n-1}$. Geometrically this factoring makes sense, as the surface cone $\mathfrak{C}(\boldsymbol x, \phi, \boldsymbol \beta)$ is the union of rays originating from $\boldsymbol x$ and emanating in the directions $\boldsymbol \alpha$ satisfying $\boldsymbol \alpha \cdot \boldsymbol \beta = \cos \phi$. Accordingly, 
\begin{equation}
 \frac{1}{\sin \phi A(\phi)} \mathcal{C}[f](\boldsymbol x, \phi, \boldsymbol \beta)   
\end{equation}
is in the range of $M$, and thus must be in the kernel of (\ref{eq: EPD_sphere}). Indeed, it was shown in \citep{Fatma} that a generalization of this result holds for the weighted CRT when the weight is an integer power of the distance from the cone vertex.

\subsection{A symmetry of the CRT} \label{sec: symmetry}
The final detail we will need to complete our description of the range of the CRT is to observe a symmetry. Namely, suppose we apply the CRT to $f\in C_{0}^{\infty}( \R^n)$ twice:
\begin{align} \label{eq: CRT_twice}
    & \mathcal{C}[\mathcal{C}[f]](\boldsymbol x, \phi,\psi, \boldsymbol \beta,\boldsymbol \gamma) \nonumber \\
    & = \sin \phi \sin \psi \int_{\mathbb{R}^{2n}} f(\boldsymbol x - \boldsymbol y - \boldsymbol u) \delta (\boldsymbol y \cdot \boldsymbol \beta - |\boldsymbol y| \cos \phi) \delta(\boldsymbol u \cdot \gamma - \cos \psi) d\boldsymbol y d\boldsymbol u.
\end{align}
Since $\mathcal{C}[f]$ is supported on a half-space, this expression makes sense when $\psi$ and $\boldsymbol \gamma$ are near $\phi$ and $\boldsymbol \beta$ respectively (See Fig. \ref{fig: DoubleCone}). Also, the role of the inner and outer CRT are interchangeable, and we have the following identity:
\begin{equation} \label{eq: symmetry}
    \frac{\partial}{\partial \boldsymbol \beta} \mathcal{C}[\mathcal{C}[f]](\boldsymbol x, \phi,\phi, \boldsymbol \beta,\boldsymbol \gamma)|_{\boldsymbol \beta = \boldsymbol \gamma} = \frac{\partial}{\partial \boldsymbol \gamma} \mathcal{C}[\mathcal{C}[f]](\boldsymbol x, \phi,\phi, \boldsymbol \beta,\boldsymbol \gamma)|_{\boldsymbol \beta = \boldsymbol \gamma}.
\end{equation}
Here, by $\frac{\partial}{\partial \boldsymbol \beta}$ we mean the gradient with respect to $\boldsymbol \beta$.
\begin{figure}[ht!]
    \centering
    \includegraphics[width=0.6\textwidth]{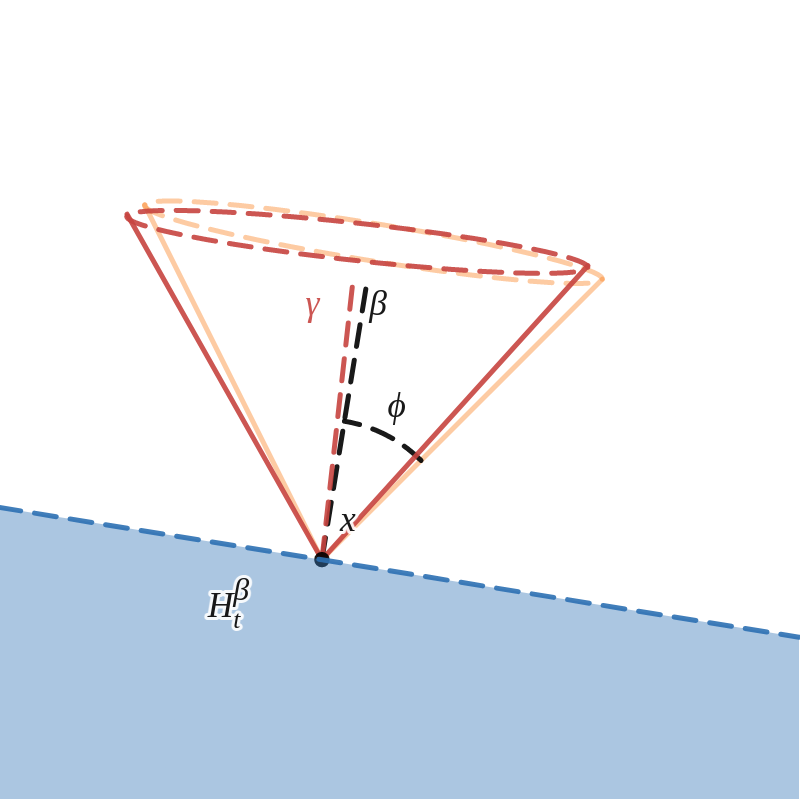}
    \caption{The CRT can be applied twice, provided the central axis direction and opening angle of the second cone are close to the first.}
    \label{fig: DoubleCone}
\end{figure}
\subsection{Range Description} \label{sec: Range_Description}
We can now obtain a characterization of the range of the full CRT.
\begin{theorem}
The function $g \in C^{\infty}(\R^n \times (0, \pi/2, \mathbb{S}^{n-1})$ can be represented as $g=\mathcal{C}[f]( \boldsymbol x, \phi, \boldsymbol \beta)$ for some $f \in C_0^{\infty}(\mathbb{R}^n)$ if and only if $g$ satisfies the following conditions:
\begin{enumerate}[(i)]
    \item For any fixed $\phi_0$, $g(\boldsymbol x, \phi_0, \boldsymbol e_n)$ satisfies the conditions of Theorems \ref{thm: even_COB} and \ref{thm: odd_COB} (depending on the parity of $n$). \label{cond: one}
    \item $\frac{1}{\sin \phi A(\phi)} g(\boldsymbol x, \phi, \boldsymbol \beta)$ is in the kernel of $\frac{\partial^2 }{\partial r^2} + \frac{A'(r)}{A(r)} \frac{\partial }{\partial r} - \Delta_{S}$ for each $\boldsymbol x$. \label{cond: two}
    \item There is a fixed $t \in \mathbb{R}$ such that $g(\boldsymbol x, \phi, \boldsymbol \beta)$ is supported in a half-space  $H_t^{\boldsymbol \beta} = \{\boldsymbol x \in \mathbb{R}^n : \boldsymbol x \cdot \boldsymbol \beta >t\}$. \label{cond: three}
    \item For all $\phi$ and $\boldsymbol \beta$, $g(\boldsymbol x, \phi, \boldsymbol \beta)$ is bounded. \label{cond: four}
    \item $\frac{\partial }{\partial \boldsymbol \beta}\mathcal{C}[g](\boldsymbol x, \phi, \phi, \boldsymbol \beta, \boldsymbol \gamma)|_{\boldsymbol \beta = \boldsymbol \gamma} =\frac{\partial }{\partial \boldsymbol \gamma}\mathcal{C}[g](\boldsymbol x, \phi, \phi, \boldsymbol \beta, \boldsymbol \gamma)|_{\boldsymbol \gamma = \boldsymbol \beta}$. \label{cond: five}
\end{enumerate}

\end{theorem}
\begin{proof}
The necessity of each of these conditions has in fact been established throughout this manuscript. Condition \ref{cond: one} was established when we considered the restricted CRT. Condition \ref{cond: two} is a consequence of the factoring of the CRT into the divergent beam and spherical mean transforms and Theorem \ref{thm: range}. Condition \ref{cond: three} was established in the case $\boldsymbol \beta = \boldsymbol e_n$, and the argument for other directions is identical.  Condition \ref{cond: four} was also established in the case $\boldsymbol \beta = \boldsymbol e_n$, and, again, the argument for other directions is identical. Condition \ref{cond: five} was established in the previous subsection. Thus we need only prove the sufficency of these conditions.

Fix $\phi_0$ and let $f\in C_{0}^{\infty}(\mathbb{R}^n)$ satisfy $g(\boldsymbol x, \phi_0,\boldsymbol e_n)=\mathcal{C}[f](\boldsymbol x, \phi_0, \boldsymbol e_n)$. The existence of $f$ is guaranteed by Theorems \ref{thm: even_COB} and \ref{thm: odd_COB}. We will show that in fact $C[f](\boldsymbol x, \phi, \boldsymbol \beta) = g(\boldsymbol x, \phi, \boldsymbol \beta)$. We first show that $\mathcal{C}[f](\boldsymbol x, \phi_0, \boldsymbol \beta)$ is the \emph{unique} solution to the boundary value problem
\begin{equation} \label{eq: Symmetry_PDE}
    \frac{\partial }{\partial \boldsymbol \beta}\mathcal{C}[u](\boldsymbol x, \phi_0, \phi_0, \boldsymbol \beta, \boldsymbol \gamma)|_{\boldsymbol \beta = \boldsymbol \gamma} =\frac{\partial }{\partial \boldsymbol \gamma}\mathcal{C}[u](\boldsymbol x, \phi_0, \phi_0, \boldsymbol \beta, \boldsymbol \gamma)|_{\boldsymbol \beta = \boldsymbol \gamma}
\end{equation}
with $u \in C^{\infty}(\R^n \times (0, \pi/2, \mathbb{S}^{n-1})$ satisfying Conditions 3 and 4 and $u(\boldsymbol x, \phi_0, \boldsymbol e_n) = g(\boldsymbol x, \phi_0, \boldsymbol e_n)$.

\begin{figure}[ht!]
    \centering
    \includegraphics[width=0.6\textwidth]{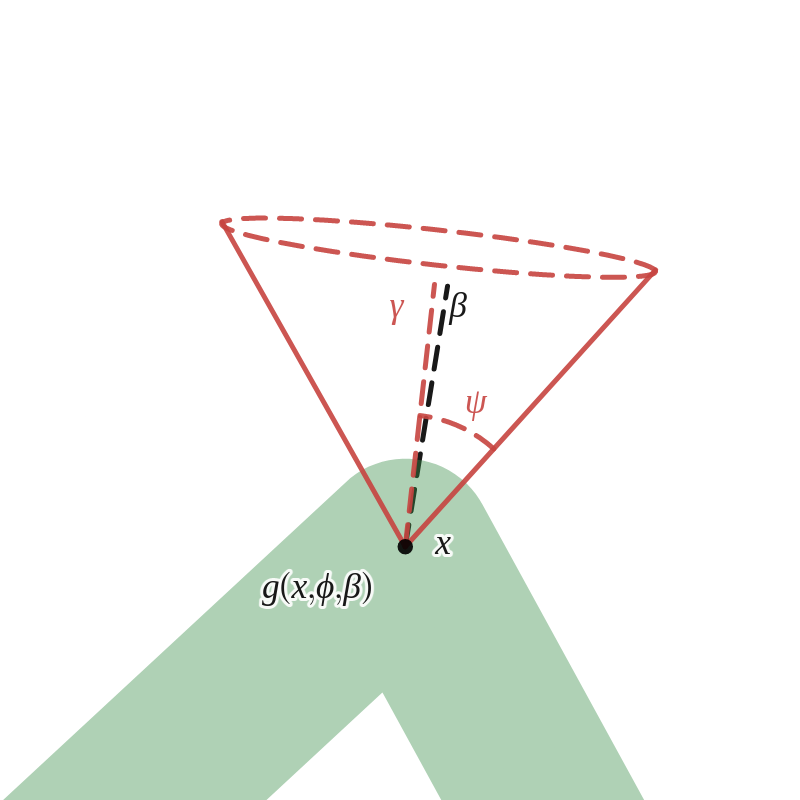}
    \caption{The CRT can be applied $g$, provided the central axis direction and opening angle of the cone are close to $\boldsymbol \beta$ and $\phi$ respectively.}
    \label{fig: DoubleConeData}
\end{figure}
Under these conditions $\mathcal{C}[u](\boldsymbol x, \phi_0, \phi_0, \boldsymbol \beta, \boldsymbol \gamma)$ will be a tempered distribution in $\boldsymbol x$. Thus the application of the Fourier transform in $\boldsymbol x$ is justified, and therefore we can compute the Fourier transform of (\ref{eq: Symmetry_PDE}) to obtain the differential equation:
\begin{equation} \label{eq: Fourier_Commutation}
    \left [ \frac{\partial}{\partial \boldsymbol \gamma} \hat{D}(\boldsymbol \omega,\phi_0,\boldsymbol \gamma) \hat{u}(\boldsymbol \omega, \phi_0, \boldsymbol \beta) \right ] _{\boldsymbol \beta = \boldsymbol \gamma}-  \left [ \hat{D}(\boldsymbol \omega,\phi_0,\boldsymbol \gamma) \frac{\partial}{\partial \boldsymbol \beta} \hat{u}(\boldsymbol \omega, \phi_0, \boldsymbol \beta) \right ]_{\boldsymbol \beta = \boldsymbol \gamma}=0
\end{equation}
where
\begin{equation}
    \hat{D}(\boldsymbol \omega, \phi_0, \boldsymbol \gamma) = p(R_{\boldsymbol \gamma} \boldsymbol \omega,\phi_0)
\end{equation}
is the Fourier transform of the cone, with
\begin{equation}
    p(\boldsymbol \omega,\phi_0) = \frac{\alpha_{n,\phi_0}i \omega_n}{(|(\omega_1,\omega_2,...,\omega_{n-1})|^2-\tan^2 \phi_0 \omega_n^2)^{n/2}}
\end{equation}
and $R_{\boldsymbol \gamma}$ is any rotation matrix such that $R_{\boldsymbol \gamma} \boldsymbol e_n = \boldsymbol \gamma.$\footnote{Due to the symmetry of the cone, any rotation mapping $\boldsymbol e_n \rightarrow \boldsymbol \gamma$ will map the cone with central axis $\boldsymbol e_n$ to the cone with central axis $\boldsymbol \gamma$. Due to the rotational invariance of the Fourier transform, its Fourier transform exhibits this same symmetry.} As discussed earlier, these distributions have analytic continuation to the complex half-space ${\Im(\boldsymbol \omega \cdot \boldsymbol \gamma) > 0}$ and are continuous at the boundary $\Im(\boldsymbol \omega \cdot \boldsymbol \gamma)=0$. Moreover, $\hat{D}$ is smooth and non-vanishing when $\Im(\boldsymbol \omega \cdot \boldsymbol \gamma) > 0$, and because (\ref{eq: Fourier_Commutation}) is linear, it admits a unique solution $\hat{u}(\boldsymbol x, \phi_0, \boldsymbol \beta)$ up to a choice of initial data $\hat{u}(\boldsymbol \omega, \phi_0 ,\boldsymbol e_n)$. In fact, the solution is given by
\begin{equation} \label{eq: Fourier_Commutation_Solution}
    \hat{u}(\boldsymbol \omega, \phi_0, \boldsymbol \beta) = \hat{U}(\boldsymbol \omega, \phi_0) \hat{D} (\boldsymbol \omega, \phi_0, \boldsymbol \beta),
\end{equation}
as can be verified by direct computation. Given the initial condition $\hat{g}(\boldsymbol \omega, \phi_0, \boldsymbol e_n)$, $\hat{U}$ is determined by
\begin{equation}
    \hat{U}(\boldsymbol \omega, \phi_0) = \frac{\hat{u}(\boldsymbol \omega, \phi_0, \boldsymbol e_n)}{\hat{D}(\boldsymbol \omega, \phi_0, \boldsymbol e_n)}=\frac{\hat{g}(\boldsymbol \omega, \phi_0, \boldsymbol e_n)}{\hat{D}(\boldsymbol \omega, \phi_0, \boldsymbol e_n)}=\frac{\widehat{\mathcal{C}[f]}(\boldsymbol \omega, \phi_0, \boldsymbol e_n)}{\hat{D}(\boldsymbol \omega, \phi_0, \boldsymbol e_n)}=\hat{f}.
\end{equation}
Furthermore, due to Condition \ref{cond: one}, $\hat{U}$ is everywhere analytic and has exponential order. Thus the solution (\ref{eq: Fourier_Commutation_Solution}) is analytic in $\Im(\boldsymbol \omega \cdot \boldsymbol \beta) > 0$. Hence, according to Conditions \ref{cond: three} and \ref{cond: four}, $\hat{u}(\boldsymbol \omega, \phi_0, \boldsymbol \beta)$ for $\boldsymbol \omega \in \R^n$ is determined by $\lim_{\Im(\boldsymbol \omega \cdot \boldsymbol \beta) \rightarrow 0}\hat{u}(\boldsymbol \omega, \phi_0, \boldsymbol \beta)$ and therefore
\begin{equation}
    \hat{u}(\boldsymbol \omega, \phi_0, \boldsymbol \beta) = \hat{f}(\boldsymbol \omega) \hat{D} (\boldsymbol \omega, \phi_0, \boldsymbol \beta) \overset{\mathcal{F}^{-1}}{\rightarrow} u(\boldsymbol x, \phi_0, \boldsymbol \beta)=\mathcal{C}[f](\boldsymbol x, \phi_0, \boldsymbol \beta).
\end{equation}
To finish the proof, we simply note that since $g(\boldsymbol x, \phi, \boldsymbol \beta)$ is in the kernel of the Euler-Poisson-Darboux operator for every point $\boldsymbol x \in \mathbb{R}^n$, by Theorems \ref{thm: injectivity} and \ref{thm: range} the map $g(\boldsymbol x, \phi_0, \boldsymbol \beta) \mapsto g(\boldsymbol x, \phi, \boldsymbol \beta)$ is bijective. Thus since $C[f](\boldsymbol x, \phi_0, \boldsymbol \beta) = g(\boldsymbol x, \phi_0, \boldsymbol \beta)$ we can conclude that $C[f](\boldsymbol x, \phi, \boldsymbol \beta) = g(\boldsymbol x, \phi, \boldsymbol \beta)$, as desired.
\end{proof}
\section{Final remarks}\label{S:concl}
In this work we determined range descriptions of the Conical Radon Transform. In the case of the restricted CRT an important connection to the d'Alembertion operator is observed, and as a result the range description depends on the parity of the dimension. Range descriptions for the restricted CRT with arbitrary opening angle and dimension are formulated. Using this description, the factorization of the CRT into the composition of divergent beam and spherical mean transforms, and symmetry of the CRT we then obtain a description of the full CRT.

Another interesting problem would be to describe the range of the transform when the vertices of the cones are restricted to a given surface, while the axes and opening angles are arbitrary. This is the natural set-up in Compton camera imaging \citep{ADHKK,bkr2020,Maxim,cargo,Fatma,TKK,X}, in which the vertices of the cones are at the detector's surface. A step toward this was made in \citep{Fatma}.

\section{Acknowledgements}\label{S:ack}

This work has been supported by the NSF grant DMS 1816430 and Texas A\&M Cyclotron Institute. The author expresses his gratitude to these agencies. Thanks also go to P.~Kuchment for posing the problem and useful comments.

\bibliography{refs}
\end{document}